\documentclass{article}
\usepackage[hidelinks]{hyperref}
\usepackage{geometry}
\usepackage{amsmath,amsthm}
\usepackage{amssymb}
\usepackage{color}
\usepackage{xspace}
\usepackage{mathrsfs}
\usepackage{framed}

\newtheorem{theorem}{Theorem}
\newtheorem{lemma}{Lemma}
\newtheorem{proposition}{Proposition}
\newtheorem{corollary}{Corollary}

\theoremstyle{definition}

\newtheorem{definition}{Definition}

\theoremstyle{remark}
\newtheorem{remark}{Remark}

\newcommand{\real}{\mathbb{R}}
\newcommand{\reals}{\mathbb{R}}
\newcommand{\vT}{v_\mathcal{T}}
\newcommand{\VT}{\hyperref[E:capitalvtdef]{V_\mathcal{T}}}
\newcommand{\IN}{I^*_N}
\newcommand{\Rx}{R_x}

\newcommand{\fullMethodName}{{Information-Theoretic Exact Method}}

\newcommand{\risk}[2]{\mathscr{\hyperref[D:risk]{R}}_{#1}{(#2)}}
\newcommand{\xrisk}[2]{\mathscr{\hyperref[D:xrisk]{X}}_{#1}{(#2)}}
\newcommand{\sz}{\lambda}

\DeclareMathOperator{\dom}{dom}
\newcommand{\oracle}{\mathcal{O}}
\newcommand{\finstances}[1]{\hyperref[D:finstances]{\mathcal{F}^{\mu, L}_{\Rx}(\real^{#1})}}
\newcommand{\invcond}{q}

\begin{document}

\title{On the oracle complexity of smooth strongly convex minimization\thanks{A. Taylor acknowledges support from the European Research Council (grant SEQUOIA 724063).This work was funded in part by the french government under management of Agence Nationale de la recherche as part of the ``Investissements d’avenir'' program, reference ANR-19-P3IA-0001 (PRAIRIE 3IA Institute).}}

% \author[*]{Yoel Drori}\affil[*]{Google Research Israel. Email: dyoel@google.com}
% \author[**]{Adrien Taylor}\affil[**]{INRIA, D\'epartement d'informatique de l'ENS, \'Ecole normale sup\'erieure, CNRS, PSL Research University, Paris, France. E-mail: adrien.taylor@inria.fr}

\author{Yoel Drori\thanks{Google Research Israel. Email: dyoel@google.com} \and Adrien Taylor\thanks{INRIA, D\'epartement d'informatique de l'ENS, \'Ecole normale sup\'erieure, CNRS, PSL Research University, Paris, France. E-mail: adrien.taylor@inria.fr}}

\maketitle

    \begin{abstract}
        We construct a family of functions suitable for establishing lower bounds on the oracle complexity of first-order minimization of smooth strongly-convex functions.
        Based on this construction, we derive new lower bounds on the complexity of strongly-convex minimization under various inaccuracy criteria. The new bounds match the known upper bounds up to a constant factor, and when the inaccuracy of a solution is measured by its distance to the solution set, the new lower bound \emph{exactly} matches the upper bound obtained by the recent \fullMethodName{} by the same authors, thereby establishing the exact oracle complexity for this class of problems.

    \textbf{Keywords: }{Strongly convex functions, oracle complexity, information-based complexity, optimal methods.}
    \end{abstract}

\section{Introduction}

In this paper, we study the performance of {deterministic} first-order methods for approximating the solution of unconstrained strongly-convex minimization problems, i.e., problems of the form
\[
    x_* \in \arg\min_{x\in \real^d} f(x),
\]
for some $d\in\mathbb{N}$, where $f$ is a strongly-convex function.
% In particular, we study the performance of first-order black-box methods, which are methods that are only allowed to gain information on the objective by evaluating the function value and gradient at points chosen by the method.

% In this paper, we study the performance of first-order methods (i.e, methods that are only allowed to gain information on the objective by evaluating the function value and gradient) for approximating the solution of strongly-convex minimization problems.

% first-order methods are treated as black-box methods, i.e., they are only allowed to gain information on the objective $f$ via a first-order oracle, $\oracle_f$ which provides the method with the function value and gradient at points requested by the algorithm.

This problem setting has the attractive property that first-order methods for solving it are simple, scalable and easy to implement, nevertheless they enjoy `fast' rates of convergence~\cite{nest-book-04}, making highly-accurate solutions relatively easy to attain.
As a result, these problems are considered `tractable' and play a key role in a wide range of applications, from machine learning, parameter estimation, computer vision and many more. These types of problems also appear as steps in methods for solving more complex problems, a property which makes efficient solutions of these problems even more important.

There has been significant progress in recent years in devising efficient methods for solving strongly-convex problems.
Suppose $f$ is a $\mu$-strongly convex function in $C^{1,1}_L(\real^d)$ (the set of continuously differentiable functions that have $L$-Lipschitz gradient),
then the classical gradient method with an appropriately chosen step size attains after $N$ iterations an approximate solution $x_N$ such that~\cite[Theorem~2.1.15]{nest-book-04}
\[
    \| x_N - x_*\| \leq \left(\frac{1-\mu/L}{1+\mu/L}\right)^N \|x_0 - x_*\|,
\]
{where here and through the rest of the paper $\|\cdot\|$ stands for the Euclidean norm.}
This rate of convergence has been improved by the celebrated Accelerated Gradient Methods~\cite{nest-book-04}, which generates a sequence of iterates converging to an optimal point at rate in the order of $O((1-\sqrt{\mu/L})^{N/2})$.
Recently, the Triple Momentum Method~\cite{van2017fastest} improved this rate even further: after $N$ iterations the method attains an approximate solution $x_N$ such that
\[
    \| x_N - x_*\| \leq \sqrt{L/\mu} \, (1-\sqrt{\mu/L})^N \|x_0-x_*\|.
\]
An additional improvement is due to the very recent \fullMethodName~\cite{taylor2021optimal}, which further improves the leading constant in the bound above.
This progress naturally raises the question: can we do even better?

A framework for formalizing this question has been pioneered by Nemirovsky and Yudin in their seminal book~\cite{nemirovsky1992information}.
Under their approach, we assume that the optimization method has access to the objective only via a first-order oracle, $\oracle_f$, that is, a subroutine which given a point in $\real^d$, returns the value of the objective and its gradient at that point. In addition, the method is provided with a starting point $x_0\in\dom(f)$ that is assumed to be ``not too far'' from an optimal solution. We call the pair  $(\oracle_f, x_0)$ a \emph{problem instance}, and for a first-order method $A$ we denote the approximate solution generated by method when applied on this problem instance by $A(\oracle_f, x_0)$.
The cost of the method is then measured by the number of calls it makes to the oracle to obtain its output.

% To put these concepts in more precise terms, consider the following unconstrained problem
% \[
%   (P) \quad f^* = \min_{x\in \real^d} f(x),
% \]
% where $f$ is a differentiable function. A \emph{first-order} optimization method is an iterative algorithm that approximates the solution of $(P)$, where it is only allowed to gain information on the objective $f$ via a first-order oracle, $\oracle_f$, that is, a subroutine which given a point in $\real^d$, returns the value of the objective and its gradient at that point.

% In this work we consider two types of accuracy guarantees:

In this work, we consider two criteria for measuring the inaccuracy of an approximate solution: \emph{absolute inaccuracy}, which quantifies the inaccuracy of an approximate solution $\xi$ for a problem instance $(\oracle_f, x_0)$ by the value of $f(\xi)-f^*$, and the \emph{distance to the solution set} which measures the inaccuracy of approximate solution $\xi$ by $\inf_{x_* \in X_*(f)} \|\xi - x_*\|$, where $X_*$ denotes the set of optimal solutions.
The \emph{efficiency estimate} of a first-order method $A$ over a given set of problem instances $\mathcal{I}$ is then defined as the worst-case value of the chosen inaccuracy measure. 
For the absolute inaccuracy measure we denote
\[
    \varepsilon(A; \mathcal{I}):=\sup_{(\oracle_f,x_0)\in \mathcal{I}}f(A(\oracle_f,x_0))-f^*,
\]
and for the distance to the solution set, we denote
\[
    \chi(A; \mathcal{I}):=\sup_{(\oracle_f,x_0)\in \mathcal{I}} \inf_{x_* \in X_*(f)} \|A(\oracle_f,x_0) - x_*\|.
\]

We can now put the main concepts addressed in this paper in formal terms:
denoting by $\mathcal{A}_N$ the set of all {deterministic} first-order methods that perform at most $N$ calls to their first-order oracle,
the \emph{minimax risk}~\cite{Guzman20151} associated with $\mathcal{I}$ is defined as the infimal efficiency estimate that a method in $\mathcal{A}_N$ can attain over $\mathcal{I}$ as a function of the computational effort $N$
\phantomsection\label{D:risk}
\[
    \risk{\mathcal{I}}{N} := \inf_{A\in\mathcal{A}_N} \varepsilon(A; \mathcal{I}).
\]
Similarly, for the distance to the solution set inaccuracy measure, we define the \emph{minimax error} by
\phantomsection\label{D:xrisk}
\[
    \xrisk{\mathcal{I}}{N} := \inf_{A\in\mathcal{A}_N} \chi(A; \mathcal{I}).
\]

The classical notion of \emph{oracle (or information-based) complexity} of the set $\mathcal{I}$ can now be identified as the inverse to the functions defined above, i.e., the minimal computational effort needed by a first-order method in order to reach a given worst-case accuracy level.
For example, under the absolute inaccuracy measure, the oracle complexity of $\mathcal{I}$ is given by
\[
    \mathscr{C}_{\mathcal{I}}(\varepsilon) := \min \{ N : \risk{\mathcal{I}}{N}\leq \varepsilon\}.
\]
In the following, we express our results using the minimax risk and minimax error functions as they prove to be better suited for expressing the dependence of the results on the dimension of the domain.
% and use the term lower complexity bound to refer to lower bounds on these functions.?

% In order to establish an upper bound on the minimax risk of a class it is sufficient to find an upper bound on the efficiency estimate of a single first-order method. On the other hand, establishing lower bounds on the minimax risk often requires a more involved analysis, as the bound needs to hold for \emph{any} first-order method.

In this paper we focus on the class of strongly convex functions with Lipschitz-continuous gradient, where the initial point is assumed to be at bounded distance from an optimal point:
\phantomsection\label{D:finstances}
\begin{align*}
    \finstances{d} := \{ (\oracle_f, x_0) : & f \text{ is a $\mu$-strongly convex function in $C^{1,1}_L(\real^d)$},\\ 
    & \| x_* - x_0\|\leq \Rx, \text{ for some } x_*\in X_*(f)\}.
\end{align*}
Lower bounds on the minimax risk and minimax error for this class of problems were derived by Nemirovski~\cite{nemirovsky1991optimality,nemirovsky1992information} and Nesterov~\cite[Theorem~2.1.13]{nest-book-04}
\begin{align*}
    \risk{\finstances{\infty}}{N}
    & \geq \frac{\mu}{2} \left(\frac{1 - \sqrt{\mu/L}}{1 +  \sqrt{\mu/L}}\right)^{2N} \Rx, \\
    \xrisk{\finstances{\infty}}{N}
    & \geq \left(\frac{1 - \sqrt{\mu/L}}{1 + \sqrt{\mu/L}}\right)^{N} \Rx,
\end{align*}
and were shown to be optimal in the sense that the number of steps required to reach a certain accuracy matches the upper bound up to a constant.
Nevertheless, a gap yet remains between the convergence rates of the upper and lower bounds, which is the purpose of this paper to close.
{Note that the bounds above were constructed using quadratic functions, making them also applicable on the subclass of strongly convex quadratic functions, thus partially explaining this gap.}

\subsection{Contribution}
% It is the goal of this work to improve the lower bounds described above, and show that the newly discovered optimization methods exactly attain the best achievable rate of convergence. In particular,
% The main results presented in this work are as follows:
The main contribution of this work is as follows:

\begin{enumerate}
    \item We present a general technique for establishing lower bounds on the oracle complexity of smooth minimization problems (i.e., problems where the gradient of the objective is Lipschitz-continuous, but the objective is not necessarily convex).
    {The technique is based on a construction that allows to smoothly extend a set of function values and gradients over the entire domain in such a way that the resulting function possesses properties making it ``hard'' to optimize by first-order methods and allows standard lower-complexity proof schemes to be applied.}
    \item We derive a lower bound on the value of $\risk{\finstances{2N+1}}{N}$ that matches the upper bound up to a constant, and derive the exact value of $\xrisk{\finstances{2N+1}}{N}$.
    \item {We present a shorter and easier to follow proof for establishing the exact minimax risk for non-strongly-convex smooth convex minimization problems, derived in~\cite{drori2017exact}.
    Notably, the proof is based on the standard set of arguments commonly used for establishing lower bounds
    and does not take advantage of special properties of the construction.}
\end{enumerate}

\section{Preliminaries}
\subsection{Strongly-convex extensions}

% The problem of \emph{interpolating} (or \emph{extending}) a strongly-convex function given the first-order information on a set of points onto a linear space has been recently been resolved by the following theorem.
% (Note that the problem of extending a function onto an arbitrary convex set remains open?)

In the problem of \emph{strongly-convex extension} (also referred to as \emph{interpolation}), given a finite or infinite set $\mathcal{T}=\{(x_i, g_i, f_i)\}_{i\in I} \subset \reals^d\times\reals^d\times\reals$ where $I$ is a set of indices, the goal is to find a $\mu$-strongly-convex function $f:\reals^d\mapsto\reals$ such that for all $i\in I$
\begin{align*}
    & f(x_i) = f_i, \\
    & \nabla f(x_i) = g_i.
\end{align*}
Note that unless stated otherwise we allow $\mu$ to be negative (where $f$ is $\mu$-strongly convex, as in the $\mu\geq 0$ case, if $f(\cdot) - \frac{\mu}{2}\|\cdot\|^2$ is convex).

Necessary and sufficient conditions for the existence of functions $f$ satisfying the conditions above were presented by Taylor, Hendrickx and Glineur in \cite{taylor2017smooth} for the case where $I$ and $d$ are finite, and independently by Azagra and Mudarra~\cite{azagra2017extension} for arbitrary sets in Hilbert spaces.
In addition, several explicit constructions of extension functions were presented~\cite{daniilidis2018explicit,drori2017exact,taylor2017smooth}.
% Note that the construction of the interpolation function itself is not unique.
In this section, we describe a different construction, which will form a main building block for this paper.

The construction is as follows.

% Several explicit constructions for an extending function were presented~\cite{taylor2017smooth,drori2017exact,daniilidis2018explicit}, however, note that these are all yield basically the same function, which is the largest possible extension.

% Note that explicit constructions, presented in~\cite{taylor2017smooth,drori2017exact} for case where $I$ is finite and~\cite{daniilidis2018explicit} for an arbitrary set, are all basically the same, yielding the largest possible extension.

% See \cite{taylor2017smooth,daniilidis2018explicit} for alternative constructions of an interpolating function.

% Note that the problem of convex extension and strongly-convex extension are basically the same 

\begin{oframed}\vspace{-10pt}
\paragraph{A strongly-convex extension}
Given $L>0$, $\mu < L$ (possibly negative), and $\mathcal{T}=\{(x_i, g_i, f_i)\}_{i\in I} \subset \reals^d\times\reals^d\times \reals$ for some finite index set $I$, we denote by
$\vT: \real^d \times \real^{I}\mapsto \real$ the quadratic function
\begin{equation*}
\begin{aligned}
    \vT(y,\alpha):=  \frac{L}{2}\|y\|^2 &- \frac{L-\mu}{2} \| y - \frac{1}{L-\mu} \sum_{i \in I} \alpha_i (g_i - \mu x_i) \|^2 \\& + \sum_{i\in I} \alpha_i \left( f_i + \frac{1}{2(L-\mu)} \|g_i - L x_i \|^2 - \frac{L}{2} \|x_i\|^2 \right),
\end{aligned}   
\end{equation*}
and by $\VT(y):\real^d\mapsto \real$ the \emph{strongly-convex extension}
\begin{equation}\label{E:capitalvtdef}
    \VT(y):=\max_{\alpha\in \Delta_{I}}\ \vT(y,\alpha),
\end{equation}
where $\Delta_I$ denotes the $|I|$-dimensional unit simplex
\[
\Delta_I:=\{ \alpha \in \real^{I}: \sum_{i\in I} \alpha_i=1,\ \alpha_i\geq 0,\ \forall i\in I\}.
\]
\end{oframed}

Before establishing the main properties of $\VT$ let us first state the standard first-order optimality conditions for the optimization problem $\VT$. For a proof see e.g.,~\cite[Example~2.1.2]{bertsekas1999nonlinear}.

% The next proposition states the first-order optimality conditions for the optimization problem $\VT$.

% We now state some immediate necessary and sufficient optimality conditions for the maximization problem $\VT(y)$. This will be fundamental in establishing the properties of $\VT$.

\begin{proposition}
Let  $L>0$, $\mu < L$ (possibly negative), and let $\mathcal{T}=\{(x_i, g_i, f_i)\}_{i\in I}\subset \reals^d\times\reals^d\times \reals$ for some finite index set~$I$.
Fix any $y\in\real^d$ and
suppose $\alpha \in \Delta_I$. Then
$\alpha$ is  optimal for $\VT(y)$ if and only if for any $k\in I$ such that $\alpha_k>0$ and any $j\in I$
\begin{equation}\label{E:forderoptalpha}
\begin{aligned}
& \langle g_j - \mu x_j, y - \frac{1}{L-\mu} \sum_i \alpha_i (g_i - \mu x_i) \rangle +  f_j + \frac{1}{2(L-\mu)} \|g_j - L x_j \|^2 - \frac{L}{2} \|x_j\|^2 \\
&\qquad \leq   \langle g_k - \mu x_k, y - \frac{1}{L-\mu} \sum_i \alpha_i (g_i - \mu x_i) \rangle + f_k + \frac{1}{2(L-\mu)} \|g_k - L x_k \|^2 - \frac{L}{2} \|x_k\|^2
\end{aligned}
\end{equation}
\end{proposition}

% \begin{proof}
% Condition~\eqref{E:forderoptalpha}
% is the first-order optimality condition for the optimization problem~\eqref{E:capitalvtdef}, see e.g.,~\cite[Example~2.1.2]{bertsekas1999nonlinear}.
% \end{proof}

The main properties of $\VT(y)$ now follow.
The theorem guarantees that $\VT$ is always $\mu$-strongly convex, has an $L$-Lipschitz gradient, and under certain conditions forms an extension of the given set of points.

\begin{theorem}\label{T:basicprop}
Let  $L>0$, $\mu < L$ (possibly negative), and let $\mathcal{T}=\{(x_i, g_i, f_i)\}_{i\in I}\subset \reals^d\times\reals^d\times \reals$ for some finite index set $I$.
Then $\VT(y)$ is $\mu$-strongly convex, in $C^{1,1}_L$ and satisfies
\[
    \VT(y) \geq f_i + \langle g_i, y - x_i \rangle + \frac{\mu}{2}\|y - x_i \|^2, \quad \forall i\in I.
\]
Furthermore, for any $i\in I$ such that
\begin{equation}\label{E:strongly_convex_interpolation_conditions}
        \frac{1}{2L} \|g_i - g_j\|^2 + \frac{\mu L}{2(L - \mu)} \|x_i - x_j - \frac{1}{L}(g_i - g_j)\|^2 \leq f_i - f_j - \langle g_j, x_i - x_j\rangle, \quad \forall j \in I,
\end{equation}
we have $\VT(x_i)=f_i$ and $\nabla \VT(x_i)=g_i$.
\end{theorem}

\begin{proof}
The result can be derived from \cite[Theorem~1]{drori2017exact} by observing that the function $\frac{L-\mu}{2}\|y\|^2 - \VT(y)$ has the same form considered by the theorem. For the sake of completeness we now give a direct proof.

To establish the smoothness and strong convexity properties, we proceed to show that both $\VT(y) - \frac{\mu}{2}\|y\|^2$ and $\frac{L}{2}\|y\|^2 - \VT(y)$ are convex.
Indeed, $\VT(y) - \frac{\mu}{2}\|y\|^2$ is a maximum of linear functions and therefore convex and $\frac{L}{2}\|y\|^2 - \VT(y)$ is convex from the convexity of $\frac{L}{2}\|y\|^2 - \vT(y,\alpha)$ and a well-known property of the infimum operator (see e.g.,~\cite[Proposition~2.22]{rockafellar2009variational}).
{Because $\VT$ is bounded between two tangent quadratic/linear functions at every point of its domain, its differentiability follows.}

Next, noting that $\alpha=e_i$ is feasible for the maximization problem $\VT$, we get
\begin{align*}
    \VT(y) 
    & \geq \vT(y, e_i) \\
    & = \frac{L}{2}\|y\|^2 -\frac{L-\mu}{2} \| y - \frac{1}{L-\mu} (g_i - \mu x_i) \|^2 + f_i + \frac{1}{2(L-\mu)} \|g_i - L x_i \|^2 - \frac{L}{2} \|x_i\|^2 \\
    & = \frac{\mu}{2}\|y\|^2 + \langle g_i - \mu x_i, y \rangle + f_i + \frac{1}{2(L-\mu)} \left(\|g_i - L x_i \|^2- \|g_i - \mu x_i\|^2\right) - \frac{L}{2} \|x_i\|^2 \\
    & = f_i + \langle g_i, y - x_i \rangle + \frac{\mu}{2}\|y - x_i \|^2.
\end{align*}

Finally, suppose~\eqref{E:strongly_convex_interpolation_conditions} holds, then in order to establish that $\VT(x_i)=f_i$ and $\nabla \VT(x_i)=g_i$ we first show that $\alpha=e_i$ is optimal for $\VT(x_i)$. Indeed, it is straightforward to verify that when $\alpha=e_i$ and $y = x_i$, conditions~\eqref{E:forderoptalpha} and~\eqref{E:strongly_convex_interpolation_conditions} are equivalent, thus $\VT(x_i) = \vT(x_i, e_i) = f_i$ and in addition, from the properties of the max operator, we have $\nabla_y \vT(x_i, e_i) \in \partial \VT(x_i)$, which implies, from the differentiability of $\VT$, that $\nabla \VT(x_i) = \nabla_y \vT(x_i, e_i) = g_i$.
\end{proof}

Note that condition~\eqref{E:strongly_convex_interpolation_conditions} necessarily holds for all $i\in I$ for every $\mu$-strongly convex function $f$ that extends $\mathcal{T}$ over the entire linear space (see e.g.,~\cite[Theorem~4]{taylor2017smooth}).
As a result, if a finite set $\mathcal{T}$ has a strongly convex extension, then $\VT$ provides one such possible extension.

% This was shown by Taylor, Hendrickx and Glineur in \cite{taylor2017smooth} for the case where $I$ is finite and the dimension of the space in finite, and independently by Azagra and Mudarra~\cite{azagra2017extension} for arbitrary sets in Hilbert spaces.

% See \cite{taylor2017smooth,daniilidis2018explicit} for alternative constructions of an interpolating function.

% Note that throughout the rest of the paper we assume $\mu\geq 0$, however, the previous theorem holds also for negative values of $\mu$ indicating a non-convex function. In particular, when $\mu=-L$, the resulting class of functions is the class of functions with $L$-Lipschitz gradients.

\subsection{Zero-chain functions and lower bounds}

In this section, we briefly review a standard technique for establishing lower complexity bounds. The technique was first introduced by Nemirovsky~\cite{nemirovsky1991optimality,nemirovsky1992information} and has been extensively used to derive lower complexity results, see~\cite{arjevani2019oracle,carmon2019lower,drori2017exact,nest-book-04,ouyang2019lower} to name a few.
Here we follow the convenient presentation of the technique due to Carmon et al.~\cite{carmon2019lower}.
The presentation, for the special case of {deterministic} first-order minimization, is based on the following definition.

% Following the convenient presentation of Carmon et al.~\cite{carmon2019lower}, we introduce the following definition.

% A common and successful technique for establishing lower complexity bounds 

% A standard argument for establishing a lower complexity bound on a class of problems proceeds as follows: \dots. This has been

% Worst-case ``zero-chain" functions have been extensively used to derive lower complexity bounds. (There was some lower complexity bound on saddle point problems that cited an early paper that used this technique.)

% One of the standard techniques for establishing lower complexity bounds is based on finding problem instances that are hard to all sensible algorithms (``worst-function in the world'')

% these ideas were generalized and standardized in the work of Carmon et al.~\cite{carmon2019lower}

\begin{definition}
A differentiable function $f:\reals^d\mapsto \reals$ is called a \emph{first-order zero-chain on $\{v_i\}_{0\leq i\leq n}$} if for all $0\leq i \leq n$
\[
    y\in \mathrm{span} \{v_0,\dots, v_{i-1}\} \Rightarrow \nabla f(y)\in \mathrm{span} \{v_0,\dots, v_i\},
\]
where we use the convention $\mathrm{span} \{\} = \{0\}$.
\end{definition}

{Note that zero-chain functions are typically defined over the canonical unit vectors, however, in the context of this paper it is useful to consider an arbitrary set which will be chosen according to the iterates and gradient values produced by the algorithm.}

Zero-chain functions are well-suited for establishing lower complexity bounds. The standard argument proceeds as follows. As a first step, we consider only the special class of first-order methods which make their first query to oracle at zero and choose the following query points only from the subspace spanned by the previous gradients seen by the method (this class of algorithms is referred to in~\cite{carmon2019lower} as zero-respecting algorithms).
This assumption plays well with the zero-chain property, as together they constrain the location of each iterate to a known subspace. By showing that all points at the final subspace are `bad' under the chosen inaccuracy measure, a bound on the performance of any algorithm in the class can be established.

The next step of the argument is to extend the bound 
% from the special class of methods considered above 
to arbitrary {deterministic} first-order methods. This is done via the ``resisting oracle'' technique pioneered by Nemirovsky and Yudin in their seminal book~\cite{nemi-yudi-book83}. This procedure requires that the problem class satisfies the following properties (see~\cite[Proposition~2]{carmon2019lower}):
\begin{enumerate}
    \item The problem class and inaccuracy measure must be invariant under orthogonal transformations.
    \item The domain of the function class must be embeddable in a domain whose dimension is arbitrarily larger.
\end{enumerate}
Suppose the conditions above hold and assume a given algorithm queries a point $x_k$ ($k\geq 0$) outside the span of the previous gradients and denote the component of the query point orthogonal to the span by $v$.
Then by applying an orthogonal transformation on the objective it is possible to find a transformed objective that has the same first-order information at the points already queried by algorithm while being non-informative in the direction $v$ (e.g., having a simple separable quadratic behavior along that direction).
%, however, for that transformed function the chosen direction is a `bad direction' that provides no additional information on the minimizer of the function.
% the new query point lie point where but the direction orthogonal to span is a `bad direction' that provides no additional information on the minimizer of the function.
Since the algorithm cannot distinguish between the two functions, it must therefore choose the same query point $x_k$ when given the transformed function, thereby at the worst-case it gains no additional information by querying at the direction $v$.

Since the problem settings which are the focus of this paper satisfy the conditions above, we have the following result.

\begin{theorem}\label{T:zeroChainBase}
Suppose $f_z:\reals^d \mapsto \reals$ is a zero-chain on $v_0,\dots,v_{n-1}$, satisfying $\|x_*\|\leq R_x$ for some $x_*\in X_*(f_z)$ then
\begin{align*}
    & \risk{\finstances{d+n}}{n} \geq \inf_{y\in \mathrm{span} \{v_0,\dots,v_{n-1}\}} f_z(y) - f_z^*, \\
    & \xrisk{\finstances{d+n}}{n} \geq \inf_{y\in \mathrm{span} \{v_0,\dots,v_{n-1}\}} d(y, X_*(f_z)).
\end{align*}
\end{theorem}
For further details and proofs, we refer the reader to \cite[Section~3]{carmon2019lower}.

\section{A family of zero-chain functions}

In this section, we present a set of easily-verifiable conditions under which $\VT$, defined in~\eqref{E:capitalvtdef}, is a zero-chain. Note that the results of this section apply to both strongly-convex and weakly-convex functions.

We start with the following technical lemma.

\begin{lemma}\label{L:zeroalpha}
Let  $L>0$, $\mu < L$ (possibly negative), and let $\mathcal{T}=\{(x_i, g_i, f_i)\}_{i\in I}\subset \reals^d\times\reals^d\times \reals$ for some finite index set~$I$. Suppose $\mathcal{T}$ satisfies
\begin{align}
    &\begin{aligned}\label{E:assumptionfLemma}
        & f_j - \frac{1}{L-\mu} \langle g_j - \mu x_j,  g_i - \mu x_i \rangle +  \frac{1}{2(L-\mu)} \|g_j - L x_j \|^2 - \frac{L}{2} \|x_j\|^2 \\&\quad \geq f_k - \frac{1}{L-\mu} \langle g_k - \mu x_k, g_i - \mu x_i \rangle + \frac{1}{2(L-\mu)} \|g_k - L x_k \|^2 - \frac{L}{2} \|x_k\|^2,
    \end{aligned}
    \quad \forall i \in I,
\end{align}
for some $j, k\in I$ and suppose there exists some $y$ such that
\[
    \langle g_j - \mu x_j, y \rangle > \langle g_k - \mu x_k, y\rangle,
\]
then for any optimal solution $\alpha^*$ of $\VT(y)$ we have $\alpha^*_k=0$.
\end{lemma}

\begin{proof}
Suppose $\alpha^*_k>0$.
Since $\sum_i \alpha^*_i = 1$, the first-order optimality condition condition~\eqref{E:forderoptalpha} can be written as
\begin{align*}
& \langle g_j - \mu x_j, y\rangle + \sum_i \alpha^*_i \left( f_j - \frac{1}{L-\mu} \langle g_j - \mu x_j,  g_i - \mu x_i \rangle +  \frac{1}{2(L-\mu)} \|g_j - L x_j \|^2 - \frac{L}{2} \|x_j\|^2 \right) \\&\quad \leq \langle g_k - \mu x_k, y \rangle+ \sum_i \alpha^*_i  \left( f_k - \frac{1}{L-\mu} \langle g_k - \mu x_k, g_i - \mu x_i \rangle + \frac{1}{2(L-\mu)} \|g_k - L x_k \|^2 - \frac{L}{2} \|x_k\|^2 \right),
\end{align*}
which yields a contradiction in view of~\eqref{E:assumptionfLemma}, the nonnegativity of $\alpha^*_i$, and the assumption on $y$.
\end{proof}

The main result of this section now follows: a set of conditions that guarantee that $\VT$ is a zero-chain function. The construction is based on a given set of $n+m$ triplets, $(x_i, g_i, f_i)$, where the first $n$ triplets will be used to form the `zero-chain space', and latter $m$ triplets are available for enforcing additional properties on $\VT$ (as will be done in the next section).

\begin{theorem}\label{T:zerochain}
Let $L>0$, $\mu < L$ (possibly negative), and let $\mathcal{T}=\{(x_i, g_i, f_i)\}_{i\in I} \subset \reals^d\times \reals^d \times \reals$ for $I = \{0,\dots,n+m-1\}$.
Suppose the following conditions hold
\begin{align}
    & x_0=0, \\
    & \langle g_i - \mu x_i, g_j - \mu x_j \rangle = \langle g_i - \mu x_i, g_k - \mu x_k \rangle, \qquad 0 \leq i < j \leq n - 1, \ k\in K_j, \label{A:gxrel} \\
    &\begin{aligned}\label{E:assumptionf}
        & f_j - \frac{1}{L-\mu} \langle g_i - \mu x_i, g_j - \mu x_j \rangle + \frac{1}{2(L-\mu)} \|g_j - L x_j \|^2 - \frac{L}{2} \|x_j\|^2 \\&\quad \geq f_k - \frac{1}{L-\mu} \langle g_i - \mu x_i, g_k - \mu x_k \rangle + \frac{1}{2(L-\mu)} \|g_k - L x_k \|^2 - \frac{L}{2} \|x_k\|^2,
    \end{aligned} \quad \begin{aligned}
            & 0\leq j\leq n-1, \\ &i\in I, \ k\in K_j,
    \end{aligned}\\
    & g_j - \mu x_j \text{ is linearly separable from } \{g_k - \mu x_k\}_{k\in K_j}, \quad 0\leq j\leq n-1, \label{A:separable}
\end{align}
where 
\[
    K_j := \{ k\in I : \text{ $g_k - \mu x_k$ is linearly independent of $\{g_0 - \mu x_0,\dots,g_j - \mu x_j\}$}\},
\]
then $\VT$ is a zero-chain on $\{g_i - \mu x_i\}_{0\leq i\leq n-1}$.
\end{theorem}

\begin{proof}
Suppose $y \in \mathrm{span} \{g_0 - \mu x_0,\dots,g_{j-1} - \mu x_{j-1}\}$ for some $0\leq j\leq n-1$. From the definition of $\VT$ it follows that
\[
    \nabla \VT(y) = \mu y + \sum_{i\in I} \alpha^*_i (g_i - \mu x_i),
\]
where $\alpha^*$ is optimal for $\nabla \VT(y)$.
Thus, in order to establish that $\VT$ is a zero-chain it is sufficient to show that $\alpha^*_k=0$ for all $k\in K_j$.

Indeed, from~\eqref{A:separable} it follows that there exists a vector $v$ such that
\[
    \langle g_j - \mu x_j, v \rangle > \langle g_k - \mu x_k, v \rangle, \quad \forall k\in K_j,
\]
then setting for some arbitrary $\varepsilon>0$
\[
    y' := y +  \varepsilon v,
\]
we have for all $k\in K_j$
\begin{align*}
    \langle g_j - \mu x_j, y' \rangle 
    & = \varepsilon \langle g_j - \mu x_j, v \rangle + \langle g_j - \mu x_j, y \rangle \\
    & = \varepsilon \langle g_j - \mu x_j, v \rangle + \langle g_k - \mu x_k, y \rangle \\
    & > \varepsilon \langle g_k - \mu x_k, v \rangle + \langle g_k - \mu x_k, y \rangle \\
    & =  \langle g_k - \mu x_k, y' \rangle,
\end{align*}
where the second equality follows from~\eqref{A:gxrel} and from the assumption on $y$.
As a result, by Lemma~\ref{L:zeroalpha}, any optimal solution for $\VT(y')$ satisfies $\alpha^*_k=0$. We conclude that
\[
    \nabla \VT(y') = \mu (y +  \varepsilon v) + \sum_{i\in I \setminus K_j} \alpha^*_i (g_i - \mu x_i).
\]
Finally, taking $\varepsilon\rightarrow 0$, then from the continuity of $\nabla \VT$ (Theorem~\ref{T:basicprop}) it follows that
\[
    \nabla \VT(y) \in  \mathrm{span} \{g_0 - \mu x_0,\dots,g_j - \mu x_j\},
\]
establishing the zero-chain property.
\end{proof}

Note that conditions~\eqref{E:strongly_convex_interpolation_conditions} were not included in the requirements of the theorem above. As a result, although the theorem guarantees that $\VT$ is a zero-chain, the function is not guaranteed to be an extension of $\mathcal{T}$. Of course, if conditions~\eqref{E:strongly_convex_interpolation_conditions} do hold for~$\mathcal{T}$, then $\VT$ will also form an extension.

% Clearly, not all zero-chain function are of the form above, 

% Note that we have not yet established that the set of requirement above is realizable by an `interesting' set of points.

\section{Lower complexity bounds on strongly convex minimization}

Theorem~\ref{T:zerochain} provides a convenient building block for establishing zero-chain based lower complexity bounds. To complete the standard argument, it remains to find a way of bounding the chosen inaccuracy measure over the span of the zero-chain base  (c.f., Theorem~\ref{T:zeroChainBase}).
For this purpose, in addition to $N$ points that form the `zero-chain space' part of the function: $x_0,\dots,x_{N-1}$, two additional points $x_N$ and $x_*$ are added together with constraints ensuring that the chosen inaccuracy measure attains its minimum over the span of the zero-chain base vectors %, $\mathrm{span} \{g_0 - \mu x_0,\dots,g_j - \mu x_j\}$
at the pair of points $x_N, x_*$.

% The idea behind the construction is to define $N$ points that will form the zero-chain part of the function: $x_0,\dots,x_{N-1}$, together with additional two points $x_N$ and $x_*$, 

% where additional constraints will be added to ensure the chosen performance measure will attain its minimum over the span of the zero-chain base vectors for the pair of points $x_N, x_*$.

% and constraints ensuring that $x_N$ will minimize the chosen performance measure over the span of the zero-chain basis and constraints ensuring $x_*$is an optimal point to $\VT$.

This construction is summarized by the following corollary.
Note that hereafter we assume convexity: $\mu\geq 0$.

\begin{corollary}\label{C:lowerBoundPEP}
Let $\Rx \in \reals$ and $0\leq \mu < L$. Suppose $\mathcal{T}=\{(x_i, g_i, f_i)\}_{i\in \IN} \subset \reals^d\times \reals^d \times \reals$ for $\IN := \{0,\dots,N,*\}$ satisfies the following conditions
\begin{align}
    & x_0=0,\ g_* = 0, \label{A:lowerBoundPEPa} \\
    & \|x_0 - x_* \| \leq \Rx, \label{A:lowerBoundPEPd} \\
    & \frac{1}{2L} \|g_j\|^2 + \frac{\mu L}{2(L - \mu)} \|x_* - x_j + \frac{1}{L} g_j\|^2 \leq f_* - f_j - \langle g_j, x_* - x_j\rangle, \quad 
    \begin{aligned}
            j \in \IN,
    \end{aligned} \label{A:lowerBoundPEPe}\\
    & \langle g_i - \mu x_i, g_j - \mu x_j \rangle = \langle g_i - \mu x_i, g_k - \mu x_k \rangle, \qquad 0 \leq i < j \leq N - 1, \ k\in K^*_j, \label{A:lowerBoundPEPf} \\
    &\begin{aligned}
        & f_j - \frac{1}{L-\mu} \langle g_i - \mu x_i, g_j - \mu x_j \rangle +  \frac{1}{2(L-\mu)} \|g_j - L x_j \|^2 - \frac{L}{2} \|x_j\|^2 \\&\quad \geq f_k - \frac{1}{L-\mu} \langle g_i - \mu x_i, g_k - \mu x_k \rangle + \frac{1}{2(L-\mu)} \|g_k - L x_k \|^2 - \frac{L}{2} \|x_k\|^2,
    \end{aligned} \quad \begin{aligned}
            & 0\leq j\leq N - 1, \\ &i\in \IN, \ k\in K^*_j,
    \end{aligned} \label{A:lowerBoundPEPg}\\
    & g_j - \mu x_j \text{ is linearly separable from } \{g_k - \mu x_k\}_{k\in K^*_j}, \qquad 0\leq j \leq N - 1, \label{A:lowerBoundPEPh}
\end{align}
with
\[
    K^*_j := \{ k\in\IN : \text{ $g_k - \mu x_k$ is linearly independent of $\{g_0 - \mu x_0,\dots,g_j - \mu x_j\}$}\},
\]
then:
\begin{enumerate}
    \item If in addition
\begin{align}
    & \langle g_N, g_i - \mu x_i \rangle = 0, \qquad 0 \leq i \leq N - 1, \label{A:fboundconditionA}\\
    & \langle g_N, x_N\rangle = 0, \label{A:fboundconditionB}
\end{align}
it follows that
\[
    \risk{\finstances{d+N}}{N} \geq f_N - f_*.
\]
\item If in addition $\mu>0$ and
\begin{align}
    & \langle x_N - x_*, g_i - \mu x_i \rangle = 0, \qquad 0 \leq i \leq N - 1, \label{A:xboundconditionA}\\
    & \langle x_N - x_*, x_N\rangle = 0, \label{A:xboundconditionB}
\end{align}
it follows that
\[
    \xrisk{\finstances{d+N}}{N} \geq \|x_N - x_*\|.
\]
\end{enumerate}
\end{corollary}

% \TODO{In the previous theorem we can remove the linear separability assumption by giving a feasible point where $g_k - \mu x_k$ are linearly independent, then use the convexity and tightness of the SDP relaxation.}

\begin{proof}
Since~\eqref{A:lowerBoundPEPe} holds, Theorem~\ref{T:basicprop} implies that $\nabla \VT(x_*) = g_* = 0$ and $\VT(x_*)=f_*$.
We conclude that $x_* \in X_*(\VT)$, ${\VT}^*=f_*$ and thus
\begin{align*}
    & \|x_0 - x_*\| \leq \Rx, \quad \text{ for $x_*\in X_*(\VT)$}.
\end{align*}

Assumptions~\eqref{A:lowerBoundPEPf}--\eqref{A:lowerBoundPEPh} imply by Theorem~\ref{T:zerochain} that $\VT$ is a zero-chain on $\{g_i - \mu x_i\}_{0\leq i\leq N-1}$ therefore, by Theorem~\ref{T:zeroChainBase}

\begin{align*}
    & \risk{\finstances{d+N}}{N} \geq \inf_{y\in \mathrm{span} \{g_i - \mu x_i\}_{0\leq i\leq N-1}} \VT(y) - \VT^*, \\
    & \xrisk{\finstances{d+N}}{N} \geq \inf_{y\in \mathrm{span} \{g_i - \mu x_i\}_{0\leq i\leq N-1}} d(y, X_*(\VT)).
\end{align*}

To establish the first claim, note that by Theorem~\ref{T:basicprop}
\[
    \VT(y) \geq f_N + \langle g_N, y - x_N \rangle + \frac{\mu}{2}\|y - x_N \|^2,
\]
it then follows from~\eqref{A:fboundconditionA} and~\eqref{A:fboundconditionB} that
\[
    \VT(y) \geq f_N + \frac{\mu}{2}\|y - x_N \|^2 \geq f_N, \quad \forall y \in \mathrm{span} \{g_i - \mu x_i\}_{0\leq i\leq N - 1},
\]
and thus get the desired bound on $\risk{\finstances{d+N}}{N}$.

The second bound follows similarly, by noting that for strongly convex functions ($\mu>0$) the minimizer $x_*$ is unique, therefore conditions~\eqref{A:xboundconditionA} and~\eqref{A:xboundconditionB} imply that
\[
    d(y, X_*(\VT)) = \|y-x_*\| \geq \|x_N - x_*\|, \quad \forall y \in \mathrm{span} \{g_i - \mu x_i\}_{0\leq i\leq N - 1}.
\]
\end{proof}

\begin{remark}
An identical argument could be used to derive bounds when the inaccuracy of an approximate solution is measured by its gradient norm. However, it is not clear how to add constraints ensuring that the gradient norm of $\VT$ becomes bounded over the relevant subspace. We therefore leave the analysis of this case for future work.
\end{remark}

% \begin{remark}
% Note that as in Theorem~\ref{T:zerochain}, a function $\VT$ satisfying the conditions of the corollary above does not necessarily form an extension of $\mathcal{T}$. 
% TODO: add motivation to this remark
% \end{remark}

We are now ready to present the main result of this section: a simple criterion for establishing lower complexity bounds on convex and strongly-convex minimization problems.

\begin{theorem}\label{T:lowerPEP}
Let $\Rx \in \reals_+$ and $0\leq \mu < L$.
Suppose the sequences $\{\gamma_i\}_{0\leq i\leq N}$ and $\{\delta_i\}_{0\leq i\leq N}$ satisfy the following conditions:
\begin{align}
    & 0\leq \frac{\mu}{L} \delta_i \leq \gamma_i \leq \delta_i, \quad i=0,\dots,N, \label{A:nonneg} \\
    & \gamma_{i+1} \delta_{i+1} \leq - (\gamma_i - \delta_i)(\gamma_i - \frac{\mu}{L} \delta_i), \quad i=0,\dots,N-1, \label{A:lowerPEPa} \\
    & \sum_{t=0}^N \delta_t^2 \leq \Rx^2, \label{E:rxdef}
\end{align}
then
\begin{align*}
    & \risk{\finstances{2N+1}}{N} \geq \frac{L^2}{2(L-\mu)} \left(2 \gamma_N \delta_N - \gamma_N^2  - \frac{\mu}{L} \delta_N^2 \right).
\end{align*}
Furthermore, if $\mu>0$ we have
\begin{align*}
    & \xrisk{\finstances{2N+1}}{N} \geq \delta_N.
\end{align*}
\end{theorem}

The proof proceeds by exhibiting a set $\mathcal{T}$ which satisfies the conditions of Corollary~\ref{C:lowerBoundPEP}. As the proof is rather technical, we postpone it to Appendix~\ref{A:Pepproof}.

Based on the theorem, we now derive several lower complexity bounds.
As a first example, the next result establishes a simple lower bound for the non-strongly convex case, $\mu=0$.
\begin{corollary}
Let $L>0, \Rx>0$ and $N\in \mathbb{N}$. Then
\begin{align*}
    \risk{\mathcal{F}^{0, L}_{\Rx}(\real^{2N+1})}{N}
    \geq \left(\frac{1}{(N+1)^2} - \frac{1}{4(N+1)^3}\right) \frac{L\Rx^2}{2} \geq \frac{3L\Rx^2}{8(N+1)^2}.
\end{align*}    
\end{corollary}
\begin{proof}
We show that the choice
\begin{align*}
    & \gamma_i := \frac{\Rx}{2(i+1)\sqrt{N+1}}, \quad i=0,\dots,N,\\
    & \delta_i := \frac{\Rx}{\sqrt{N+1}}, \quad i=0,\dots,N,
\end{align*}
satisfies the conditions of Theorem~\ref{T:lowerPEP}.
First,~\eqref{A:nonneg} and~\eqref{E:rxdef} are trivial to verify, therefore only need to verify~\eqref{A:lowerPEPa}, which holds since for $i=0,\dots,N-1$ we have
\begin{align*}
    \gamma_{i+1} \delta_{i+1}
    & = \frac{1}{2(i+2)}\frac{1}{N+1} \Rx^2 \\
    & \leq \frac{2 i + 1}{4(i+1)^2}\frac{1}{N+1} \Rx^2 \\
    & = \left(\frac{1}{\sqrt{N+1}} - \frac{1}{2(i+1)\sqrt{N+1}}\right) \frac{1}{2(i+1)\sqrt{N+1}} \Rx^2 \\
    & = -(\gamma_i - \delta_i) \gamma_i.
\end{align*}
Finally, 
\[
    2 \gamma_N \delta_N - \gamma_N^2 = \frac{\Rx^2}{(N+1)^2} - \frac{\Rx^2}{4(N+1)^3},
\]
which completes the proof.
\end{proof}

The exact lower bound for the non-strongly convex case, derived in~\cite{drori2017exact}, can be also attained by Theorem~\ref{T:lowerPEP}, although with some additional effort. See an outline of the proof in Appendix~\ref{S:exactBound}. 
Note, however, that the lower bound derived in~\cite{drori2017exact} is slightly stronger, as it makes a weaker requirements on the dimension $d$ of the domain of the problem, i.e., it establishes the results for problems over $\reals^d$ with $d\geq N+1$ instead of $d \geq 2N+1$ as guaranteed by Theorem~\ref{T:lowerPEP} {(the proof in~\cite{drori2017exact} takes advantage of properties of the constructed function that do not hold for general zero-chain functions)}.

Turning to the strongly-convex $\mu>0$ case, a simple bound can be obtained via Theorem~\ref{T:lowerPEP} as follows.
\begin{corollary}
Let $L, \Rx>0$, $0<\mu<L$ and $N\in \mathbb{N}$,
then
\begin{align*}
    \risk{\finstances{2N+1}}{N}
    & \geq \mu\frac{2  - \sqrt{\mu / L}}{1 + \sqrt{\mu / L }} \left(1-\sqrt{\frac{\mu }{L}}\right)^{2N} \Rx^2 \geq \frac{\mu}{2} \left(1-\sqrt{\frac{\mu}{L}}\right)^{2N} \Rx^2, \\
    \xrisk{\finstances{2N+1}}{N}
    & \geq \sqrt{1- \left(1-\sqrt{\frac{\mu }{L}}\right)^2 } \left(1-\sqrt{\frac{\mu }{L}}\right)^{N} \Rx \geq \sqrt[4]{\frac{\mu}{L}} \left(1-\sqrt{\frac{\mu }{L}}\right)^{N} \Rx.
\end{align*}
\end{corollary}

\begin{proof}
Consider the sequences defined by
\begin{align*}
    & \gamma_j := \sqrt{\frac{\mu}{L}} \sqrt{1- \left(1-\sqrt{\frac{\mu }{L}}\right)^2 } \left(1-\sqrt{\frac{\mu }{L}}\right)^{j} \Rx \\
    & \delta_j := \sqrt{1- \left(1-\sqrt{\frac{\mu }{L}}\right)^2 } \left(1-\sqrt{\frac{\mu }{L}}\right)^{j} \Rx,
\end{align*}
then~\eqref{A:nonneg} holds since $\mu/L\leq \sqrt{\mu/L}\leq 1$ and
\eqref{A:lowerPEPa} holds with equality since
\begin{align*}
    & - (\gamma_j - \delta_j)(\gamma_j - \frac{\mu}{L} \delta_j) \\
    & = - \left(\sqrt{\frac{\mu}{L}} - 1\right)\left(\sqrt{\frac{\mu}{L}} - \frac{\mu}{L}\right) \left(1- \left(1-\sqrt{\frac{\mu }{L}}\right)^2 \right) \left(1-\sqrt{\frac{\mu }{L}}\right)^{2j} \Rx^2 \\
    & = \sqrt{\frac{\mu}{L}} \left(1 - \sqrt{\frac{\mu}{L}}\right)^2 \left(1- \left(1-\sqrt{\frac{\mu }{L}}\right)^2 \right) \left(1-\sqrt{\frac{\mu }{L}}\right)^{2j} \Rx^2 \\
    & = \gamma_{j+1}\delta_{j+1}.
\end{align*}
Furthermore,
\begin{align*}
    \sum_{k=0}^N \delta_k^2 \leq \sum_{k=0}^\infty \delta_k^2 = \Rx^2,
\end{align*}
hence the conditions of Theorem~\ref{T:lowerPEP} are satisfied.
Finally,
\begin{align*}
    & \frac{L^2}{2(L-\mu)} \left( 2 \gamma_N \delta_N - \gamma_N^2 - \frac{\mu}{L} \delta_N^2 \right) \\
    & = \frac{L^2 \Rx^2}{2(L-\mu)} \left( 2 \sqrt{\mu L} - \frac{\mu}{L} -  \frac{\mu}{L} \right) \left(1- \left(1-\sqrt{\frac{\mu }{L}}\right)^2 \right) \left(1-\sqrt{\frac{\mu }{L}}\right)^{2N} \\
    & = \frac{\mu L \Rx^2}{L-\mu} \left(\sqrt{L/\mu} - 1\right) \left(2 \sqrt{\frac{\mu }{L}} -\frac{\mu }{L} \right) \left(1-\sqrt{\frac{\mu }{L}}\right)^{2N} \\
    & = \mu \frac{1 - \sqrt{\mu / L }}{1 - \mu/L} \left(2  - \sqrt{\frac{\mu }{L}} \right) \left(1-\sqrt{\frac{\mu }{L}}\right)^{2N} \Rx^2 \\
    & = \mu\frac{2  - \sqrt{\mu / L}}{1 + \sqrt{\mu / L }} \left(1-\sqrt{\frac{\mu }{L}}\right)^{2N} \Rx^2,
\end{align*}
completing the proof.
\end{proof}

As in the non-strongly convex case, the previous bounds can be further improved at the cost of a greater proof complexity. In the following result, we give such an improvement for the distance to the solution set inaccuracy measure.
The bound \emph{exactly} matches the worst-case performance attained by the \fullMethodName{}~\cite{taylor2021optimal}, thus, as the upper and lower bounds on the complexity of the class are equal, it follows that both bounds are exact and cannot be further improved.
\begin{corollary}\label{C:exactLower}
Let $\Rx>0$, $0<\mu<L$ and $N\in \mathbb{N}$, and define the sequence $\{\sz_i\}$ by,
\begin{equation}\label{D:lambdadef}
\begin{aligned}
    & \sz_0=\sqrt{\mu/L} \\
    & \sz_{i+1}  = \frac{1-\sqrt{\mu/L - (1 - \mu/L) \sz_i^2}}{1+ \sz_i^2} \sz_i
    ,\quad i=0,1,\dots,
\end{aligned}    
\end{equation}
then $\mu/L - (1 - \mu/L) \sz_i^2 \geq 0$ for all $i$, and
\begin{align*}
    % & \risk{\finstances{2N+1}}{N} \geq \sz_N^2 \frac{L\Rx^2}{2}, \\
    & \xrisk{\finstances{2N+1}}{N} \geq \frac{\sz_N \Rx}{\sqrt{\mu/L}}. % \geq \left(1-\sqrt{\frac{\mu }{L}}\right)^{N} \Rx.
\end{align*}
\end{corollary}
The proof establishing the corollary is detailed in Appendix~\ref{A:exactLowerProof}.

% Another consequence of this result, it that the Triple-Momentum method~\cite{van2017fastest} attains the best possible rate of convergence for minimizing strongly-convex functions.

\section{Conclusion}
In this work, we presented a construction of a family of zero-chain functions suitable for establishing lower complexity bounds for smooth problems, and showed how the construction can be used to derive lower complexity bounds on the class of strongly-convex minimization. Based on this result, we obtained the following bound on the minimax risk
\begin{align*}
    \risk{\finstances{2N+1}}{N}
    & \geq \max \left( 2\frac{\mu}{L}\frac{2  - \sqrt{\mu / L}}{1 + \sqrt{\mu / L }} \left(1-\sqrt{\frac{\mu }{L}}\right)^{2N}, \frac{1}{\theta_N^2}\right) \frac{L\Rx^2}{2} \\
    & \geq \max\left(\frac{\mu}{L} \left(1-\sqrt{\frac{\mu}{L}}\right)^{2N}, \frac{3}{4(N+1)^2}\right) \frac{L\Rx^2}{2},
\end{align*}
for any $0\leq \mu < L$ with $\theta_N$ defined as in~\eqref{D:thetadef},
and the bound
\begin{align*}
    \xrisk{\finstances{2N+1}}{N}
    & \geq \frac{\sz_N \Rx}{\sqrt{\mu/L}} \geq \sqrt[4]{\frac{\mu}{L}} \left(1-\sqrt{\frac{\mu }{L}}\right)^{N} \Rx,
\end{align*}
for $0< \mu < L$, 
where $\sz_N$ is defined in~\eqref{D:lambdadef}.
These bounds were then shown to exhibit the optimal rate of convergence for this class, and in the case of the distance to the solution set inaccuracy measure the bound is exact.

Note that the results only apply to the absolute inaccuracy and distance to the solution set measures (and possibly other measures solely based on these two).
An open question that remains, is the possibility of using the same construction in obtaining lower bounds when the inaccuracy measure is chosen to be the norm of the gradient. This question naturally arises when attempting to generalize the above results to the class of weakly-convex functions.
We leave the analysis of this case for future research.

Another open question that remains is the \emph{exact} oracle complexity of strongly-convex minimization under the absolute inaccuracy measure.
Preliminary numerical experiments suggest that the upper bound derived in \cite{taylor2021optimal} does not match the best bound attainable by Theorem~\ref{T:lowerPEP}, leading us to raise the conjecture that the best bound attainable by Theorem~\ref{T:lowerPEP} for absolute inaccuracy performance measure is not tight for $\mu>0$, and that a more refined approach is required to attain the exact value of the minimax risk function for this case.

\appendix
\section*{Appendix}

\section{Proof of Theorem~\ref{T:lowerPEP}}\label{A:Pepproof}
The proof proceeds by showing that the requirements of Corollary~\ref{C:lowerBoundPEP} are satisfied taking the set $\mathcal{T}=\{(\hat x_i, \hat g_i, \hat f_i)\}_{i\in \IN}$ defined by
\begin{align*}
    & \hat x_j := - \sum_{t=0}^{j-1} \delta_t e_t, \quad j=0,\dots,N, \\
    & \hat g_j := L \gamma_j e_j, \quad j=0,\dots,N, \\
    & \hat f_j := \frac{L^2}{2(L-\mu)} \left( 2 \gamma_j \delta_j - \gamma_j^2 - \frac{\mu}{L} \sum_{t=j}^N \delta_t^2 \right), \quad j=0,\dots,N, \\
    & \hat x_* := -\sum_{t=0}^N \delta_t e_t, \\
    & \hat g_* := 0, \\
    & \hat f_* := 0.
\end{align*}

To establish the requirements of Corollary~\ref{C:lowerBoundPEP}, first note that \eqref{A:lowerBoundPEPa}, \eqref{A:lowerBoundPEPf}, and \eqref{A:fboundconditionA}--\eqref{A:xboundconditionB} are immediate from the construction, and \eqref{A:lowerBoundPEPd} follows directly from~\eqref{E:rxdef}.
Next, to establish \eqref{A:lowerBoundPEPe}, we substitute the choice of $x_i, g_i$ above to get
\begin{align*}
    \frac{L}{2} \gamma_j^2 + \frac{\mu L}{2(L - \mu)} \left( (\gamma_j - \delta_j)^2 + \sum_{i=j+1}^N \delta_i^2  \right) \leq - \hat f_j + L \gamma_j \delta_j,
\end{align*}
which follows since
\[
    \hat f_j = \frac{L^2}{2(L - \mu)} \left(-\frac{L-\mu}{L} \gamma_j^2 -\frac{\mu}{L}(\gamma_j - \delta_j)^2 + \frac{2(L-\mu)}{L} \gamma_j \delta_j - \frac{\mu}{L} \sum_{i=j+1}^N \delta_i^2 \right).
\]

Linear separability of $\hat g_j - \mu \hat x_j$ and $\{\hat g_k - \mu \hat x_k\}_{k\in K^*_j}$~\eqref{A:lowerBoundPEPh} can be established by considering the following two cases: First, if $L \gamma_j \neq \mu \delta_j$, taking $v=e_j$ gives $\langle \hat g_j - \mu \hat x_j, v\rangle = L \gamma_j$ while for $k>j$ and $k=*$, we have $\langle \hat g_k - \mu \hat x_k, v\rangle = \mu \delta_j$, establishing separability.
Otherwise, when $L \gamma_j = \mu \delta_j$,
\[
(\hat g_k - \mu \hat x_k) - (\hat g_j - \mu \hat x_j) =
\begin{cases}
    L \gamma_k e_k + \mu \sum_{t=j+1}^{k-1} \delta_t e_t, & k\neq *, \\
    \mu \sum_{t=j+1}^{N} \delta_t e_t, & k = *,
\end{cases}
\]
thus taking $v= \sum_{t>j} e_t$ we get 
\[
    \langle(\hat g_k - \mu \hat x_k) - (\hat g_j - \mu \hat x_j), v\rangle =
    \begin{cases}
    L \gamma_k + \mu \sum_{t=j+1}^{k-1} \delta_t, & k\neq *, \\
    \mu \sum_{t=j+1}^{N} \delta_t, & k = *,
    \end{cases}
\]
which must be positive since $\gamma_k$ and $\{\mu \delta_i\}_{j< i<k}$ (or $\{\mu \delta_i\}_{j< i\leq N}$ when $k=*$) are nonnegative and are not all zeros, as the definition of $K^*_j$ implies that $\hat g_k - \mu \hat x_k \neq \hat g_j - \mu \hat x_j$.

We complete the proof of Theorem~\ref{T:lowerPEP} by establishing that the choice for $\{(\hat x_i, \hat g_i, \hat f_i)\}$
satisfies~\eqref{A:lowerBoundPEPg}. Note that from the symmetry in~\eqref{A:lowerBoundPEPg}, it is sufficient to verify it for $j=0,\dots,N-1, k=j+1$ and for $j=N, k=*$, for all $i\in \IN$.

\subsection{Case 1: $k=j+1$}
We start by considering the case $0\leq j\leq N-1$, $k=j+1$.
By the choice of $\hat x_i$, $\hat g_i$ it follows that $\langle \hat g_i, \hat x_i\rangle=0$ for all $i\in \IN$, thus~\eqref{A:lowerBoundPEPg} can be simplified to
\begin{align*}
\begin{aligned}
        & \frac{1}{L-\mu} \langle \hat g_i - \mu \hat x_i, \hat g_{j+1} - \hat g_j - \mu (\hat x_{j+1}  - \hat x_j) \rangle \\&\quad \geq \hat f_{j+1} - \hat f_j + \frac{1}{2(L-\mu)}( \|\hat g_{j+1} \|^2 - \|\hat g_j \|^2) + \frac{L\mu}{2(L-\mu)}(\|\hat x_{j+1}\|^2 - \|\hat x_j\|^2),
\end{aligned} \quad \begin{aligned}
            & 0\leq j\leq N - 1, \\ &i\in \IN,
\end{aligned}
\end{align*}
Substituting the choice of $\{(\hat x_i, \hat g_i, \hat f_i)\}$ we get
\begin{align*}
\begin{aligned}
        & \frac{L^2}{L-\mu} \langle \gamma_i e_i + \frac{\mu}{L} \sum_{t=0}^{i-1} e_t \delta_t , \gamma_{j+1} e_{j+1} - \gamma_j e_j  + \frac{\mu}{L} e_j \delta_j \rangle \\&\quad \geq \frac{L^2}{L-\mu} (-\gamma_j \delta_j + \gamma_{j+1} \delta_{j+1} + \frac{\mu}{L} \delta_j^2),
\end{aligned} \quad \begin{aligned}
            & 0\leq j\leq N - 1, \\ &i\in \IN,
\end{aligned}
\end{align*}
where for ease of notation we set $\gamma_*:=0$. Then from~\eqref{A:lowerPEPa} and the assumption $\mu<L$ we get that in order to establish the expression above it is sufficient to verify that the following inequality holds
\begin{align*}
\begin{aligned}
        & \langle \gamma_i e_i + \frac{\mu}{L} \sum_{t=0}^{i-1} e_t \delta_t , \gamma_{j+1} e_{j+1} - \gamma_j e_j  + \frac{\mu}{L} e_j \delta_j \rangle \geq \gamma_j (\frac{\mu}{L} \delta_j - \gamma_j),
\end{aligned} \quad \begin{aligned}
            & 0\leq j\leq N - 1, \\ &i\in \IN.
\end{aligned}
\end{align*}
We consider the following four cases:

\paragraph{1.} $i<j$. Here, the left-hand side of the inequality is zero, and the inequality follows directly from the non-negativity constraints~\eqref{A:nonneg}.

\paragraph{2.} $i=j$. For this case, the inequality above becomes
    \begin{align*}
         & -\gamma_j^2 + \frac{\mu}{L} \gamma_j\delta_j \geq \gamma_j (\frac{\mu}{L} \delta_j - \gamma_j),
    \end{align*}
    which trivially holds with equality.
    
\paragraph{3.} $i=j+1$. Here, we need to establish
    \begin{align*}
        & \gamma_{j+1}^2 - \frac{\mu}{L} \gamma_j \delta_j + \frac{\mu^2}{L^2} \delta_j^2 \geq \gamma_j (\frac{\mu}{L} \delta_j - \gamma_j),
    \end{align*}
    which holds as it reduces to
    \[
        \gamma_{j+1}^2 + (\gamma_j - \frac{\mu}{L} \delta_j)^2 \geq 0.
    \]
    
\paragraph{4.} $i>j+1$ or $i=*$. We have
    \begin{align*}
         & \frac{\mu}{L} \gamma_{j+1}\delta_{j+1} - \frac{\mu}{L} \gamma_j \delta_j + \frac{\mu^2}{L^2} \delta_j^2 \geq \gamma_j (\frac{\mu}{L} \delta_j - \gamma_j)
    \end{align*}
    which reduces to
    \[
        \frac{\mu}{L} \gamma_{j+1}\delta_{j+1} + (\gamma_j - \frac{\mu}{L} \delta_j)^2 \geq 0
    \]
    and follows from the nonnegativity assumptions~\eqref{A:nonneg}.

\subsection{Case 2: $j=N$, $k=*$}
For this case, $\gamma_*=0$ and therefore \eqref{A:lowerBoundPEPg} can be simplified to
\begin{align*}
\begin{aligned}
        & \frac{1}{L-\mu} \langle \hat g_i - \mu \hat x_i, - \hat g_N - \mu (\hat x_* - \hat x_N) \rangle \\&\quad \geq - \hat f_N  - \frac{1}{2(L-\mu)} \|\hat g_N \|^2 + \frac{L\mu}{2(L-\mu)}(\|\hat x_*\|^2 - \|\hat x_N\|^2),
\end{aligned} \quad \begin{aligned}
            & i\in \IN.
\end{aligned}
\end{align*}
Substituting $\hat x_i$, $\hat g_i$ and $\hat f_i$ with the values chosen above, we get
\begin{align*}
\begin{aligned}
        & \frac{L^2}{L-\mu} \langle \gamma_i e_i + \frac{\mu}{L} \sum_{t=0}^{i-1} e_t \delta_t, (\frac{\mu}{L} \delta_N - \gamma_N) e_N  \rangle \\ 
        & \quad \geq -\frac{L^2}{2(L-\mu)} \left( 2 \gamma_N \delta_N - \gamma_N^2 - \frac{\mu}{L} \delta_N^2 \right) - \frac{L^2}{2(L-\mu)} \gamma_N^2 + \frac{L \mu}{2(L-\mu)} \delta_N^2,
\end{aligned} \quad \begin{aligned}
            & i\in \IN.
\end{aligned}
\end{align*}
We therefore need to establish
\begin{align*}
\begin{aligned}
        & \langle \gamma_i e_i + \frac{\mu}{L} \sum_{t=0}^{i-1} e_t \delta_t, (\frac{\mu}{L} \delta_N - \gamma_N) e_N  \rangle \geq -\gamma_N \delta_N + \frac{\mu}{L}\delta_N^2,
\end{aligned} \quad \begin{aligned}
            & i\in \IN.
\end{aligned}
\end{align*}

We consider the following three cases:

\paragraph{1.} $i<N$.
Here the right-hand side of the expression is zero and the inequality follows from~\eqref{A:nonneg}.

\paragraph{2.} $i=N$. In this case, we need to show 
\begin{align*}
    & \gamma_N  \left(\frac{\mu}{L} \delta_N - \gamma_N \right) \geq -\gamma_N \delta_N + \frac{\mu}{L}\delta_N^2,
\end{align*}
which follows from~\eqref{A:nonneg}.

\paragraph{3.} $i=*$.
Finally, need to show 
\begin{align*}
    & \frac{\mu}{L} \delta_N \left( \frac{\mu}{L} \delta_N - \gamma_N \right) \geq -\gamma_N \delta_N + \frac{\mu}{L}\delta_N^2,
\end{align*}
which follows since $\delta_N\geq 0$, $\frac{\mu}{L} \delta_N - \gamma_N\leq 0$, and $\frac{\mu}{L} \leq 1$.

\section{The exact bound for $\mu=0$}\label{S:exactBound}
% The exact lower bound for the non-strongly convex case~\cite{drori2017exact} can also be recovered from Theorem~\ref{T:lowerPEP}. Note that this is a slightly weaker version, as \cite{drori2017exact} requires the dimension of the space to be $d\geq N+1$.
\begin{corollary}
Let $L>0, \Rx>0$ and $N\in \mathbb{N}$. Then
\begin{align*}
    \risk{\mathcal{F}^{0, L}_{\Rx}(\real^{2N+1})}{N}
    \geq \frac{L\Rx^2}{2\theta_N^2},
\end{align*}
where
\begin{equation}\label{D:thetadef}
\begin{aligned}
    & \theta_0 = 1, \\
    & \theta_i = \frac{1+\sqrt{1 + 4 \theta_{i-1}^2}}{2}, \quad i=1,\dots,N-1, \\
    & \theta_N =\frac{1+\sqrt{1 + 8 \theta_{N-1}^2}}{2}.
\end{aligned}    
\end{equation}
\end{corollary}

\begin{proof}
Let
\begin{align*}
    & \zeta_i=\frac{2\theta_i}{2\theta_i-1}\zeta_{i+1}, \quad i=0,\dots,N-1,\\
    & \zeta_N = \frac{\theta_N}{\theta_N-1}\zeta_{N+1}, \\
    & \zeta_{N+1} = \frac{\theta_N-1}{\theta_N^2(2\theta_N-1)} \Rx^2, \\
    & \zeta_{N+2} = 0,
\end{align*}
and set
\begin{align*}
    \gamma_i := \sqrt{\zeta_i - \zeta_{i+1}}, \quad i=0,\dots,N,\\
    \delta_i := \frac{\zeta_i}{\sqrt{\zeta_i - \zeta_{i+1}}}, \quad i=0,\dots,N.
\end{align*}

To complete the proof, it is sufficient to show that
\begin{align*}
    & \gamma_{i+1} \delta_{i+1} =  -(\gamma_i - \delta_i) \gamma_i, \quad i=0,\dots,N-1, \\
    & \sum_{i=0}^N \delta_i^2 \leq 1,
\end{align*}
and 
\[
    2 \gamma_N \delta_N - \gamma_N^2 = \frac{1}{2\theta_N^2},
\]
which are established as part of the proof of \cite[Lemma~3]{drori2017exact}.
\end{proof}

\section{Proof of Corollary~\ref{C:exactLower}} \label{A:exactLowerProof}
For the sake of conciseness, let us denote by $\invcond:=\mu/L$ the inverse condition number, which we assume to be in the open interval $(0,1)$.
We begin the proof by showing that the sequence $\sz_i$ is well-defined. This follows by a simple inductive argument: suppose $0\leq \sz_i \leq \sqrt{\invcond}$, then $\invcond^2 \leq \invcond - (1 - \invcond) \sz_i^2\leq \invcond$ and we therefore have
\begin{equation}\label{E:szUpperLowerBound}
    0 \leq \frac{1-\sqrt{\invcond}}{1 + \invcond} \sz_i\leq \sz_{i+1} 
    = \frac{1-\sqrt{\invcond - (1 - \invcond) \sz_i^2}}{1+ \sz_i^2} \sz_i \leq (1-\invcond) \sz_i \leq \sqrt{\invcond}.
\end{equation}
In particular, $\invcond - (1 - \invcond) \sz_i^2\geq 0$ for all $i$ and thus the sequence $\sz_i$ is well-defined.
% In addition, from~\eqref{E:szUpperLowerBound} and $\sz_0=\sqrt{q}$ we get that $\sz_i \geq (1-\sqrt{q})^i \sqrt{\invcond}$, which establishes the right-hand side of the lower bound in the claim.

The main part of the proof also proceeds by induction.
Assume without loss of generality that $\Rx=1$ and suppose $\{\gamma_i\}_{0\leq i\leq N}$, $\{\delta_i\}_{0\leq i\leq N}$ satisfy
\begin{align*}
    & 0\leq \invcond \delta_i \leq \gamma_i \leq \delta_i, \quad i=0,\dots,N, \\
    & \gamma_{i+1} \delta_{i+1} = - (\gamma_i - \delta_i)(\gamma_i - \invcond \delta_i), \quad i=0,\dots,N-1, \\
    & \sum_{i=0}^N \delta_i^2 = 1, \\
    & \gamma_{N} = \invcond^{1/2} \sz_N, \quad \delta_N = \invcond^{-1/2} \sz_N.
\end{align*}
We proceed to show that $\{\hat \gamma_i\}_{0\leq i\leq N+1}$, $\{\hat \delta_i\}_{0\leq i\leq N+1}$
defined by
\begin{align*}
    \hat \gamma_i :=& \sqrt{\alpha_{N+1}} \gamma_i, \quad 0\leq i\leq N - 1, \\
    \hat \gamma_N 
    % :=& \sqrt{\frac{1}{2} \left((1 + \invcond) \alpha_{N+1} \sz_N^2 - \sz_{N+1}^2 +\sqrt{\left((1 + \invcond) \alpha_{N+1} \sz_N^2 - \sz_{N+1}^2\right)^2 - 4 \invcond  \alpha_{N+1}^2 \sz_N^2^2}\right)} \\
    :=& \sqrt{\invcond  \left(\alpha_{N+1} \sz_N^2 + \alpha_{N+1} -1\right)} \\
    \hat \gamma_{N+1} :=& \invcond^{1/2} \sz_{N+1}, \\
    \hat \delta_i :=& \sqrt{\alpha_{N+1}} \delta_i, \quad 0\leq i\leq N - 1, \\
    \hat \delta_N 
    % :=& \sqrt{\frac{1}{2\invcond} \left( (1 + \invcond) \alpha_{N+1} \sz_N^2 - \sz_{N+1}^2 - \sqrt{\left((1 + \invcond) \alpha_{N+1} \sz_N^2 - \sz_{N+1}^2 \right)^2 - 4 \invcond \alpha_{N+1}^2 \sz_N^2^2}\right)} \\
    :=& \sqrt{\invcond^{-1}\left(\alpha_{N+1} \sz_N^2 - \sz_{N+1}^2 - \invcond(\alpha_{N+1} - 1)\right)} \\
    \hat \delta_{N+1} :=& \invcond^{-1/2} \sz_{N+1},
\end{align*}
with
\begin{align*}
    \alpha_{N+1} 
    %:= \frac{1}{1 + \sz_N^2} \left(\frac{\sz_N^2}{\sqrt{\invcond - (1 - \invcond) \sz_N^2}}+ 1 \right) 
    % := 1 + \frac{\sqrt{\sz_N^2 \sz_{N+1}^2}}{\sqrt{\invcond - (1 - \invcond) \sz_N^2}}
    % = 1 + \frac{(1-\invcond ) \sz_N^2}{\invcond -(1-\invcond ) \sz_N^2+\sqrt{\invcond -(1-\invcond ) \sz_N^2}}
    := \sqrt{\frac{\invcond - \sz_{N+1}^2}{\invcond - (1 - \invcond) \sz_N^2}}
\end{align*}
satisfy
\begin{align*}
    & 0\leq \invcond \hat \delta_i \leq \hat \gamma_i \leq \hat \delta_i, \quad i=0,\dots,N+1, \\
    & \hat \gamma_{i+1} \hat \delta_{i+1} = - (\hat \gamma_i - \hat \delta_i)(\hat \gamma_i - \invcond \hat \delta_i), \quad i=0,\dots,N, \\
    & \hat \gamma_{N+1} \hat \delta_{N+1} = \sz_{N+1}^2, \\
    & \sum_{i=0}^{N+1} \hat \delta_i^2 = 1.
\end{align*}

We show that the expressions above are well-defined in \S\ref{S:wellDefinedAlphaGammaDelta} and establish the relations above in~\S\ref{S:gammaDeltaIneq}--\S\ref{S:deltaSumEq}.
For the base of the induction, we take $\gamma_0 = \invcond^{-1}$, $\delta_0=1$.
Finally, the claim of Corollary~\ref{C:exactLower} follows directly from Theorem~\ref{T:lowerPEP}.

\subsection{$\alpha_{N+1}$, $\hat \gamma_N$ and $\hat \delta_N$ are well-defined}\label{S:wellDefinedAlphaGammaDelta}
The well-definiteness of $\alpha_{N+1}$ follows immediately from $0\leq \sz_N^2,\sz_{N+1}^2 \leq \invcond$.

Next, we show 
\begin{align*}
    & \alpha_{N+1} \sz_N^2 + \alpha_{N+1} - 1 \geq 0, \\
    & \alpha_{N+1} \sz_N^2 - \sz_{N+1}^2 - \invcond (\alpha_{N+1} - 1)\geq 0,
\end{align*}
and along the way, establish convenient expressions for $\hat \gamma_N$ and $\hat \delta_N$.
First, let us show the following identity
\begin{align*}
    \invcond - \sz_{N+1}^2 
    % & = \invcond - \left(\frac{1-\invcond}{1+ \sqrt{\invcond - (1 - \invcond) \sz_N^2}}\right)^2 \sz_N^2 \\
    & = \invcond - \left(\frac{1-\sqrt{\invcond - (1 - \invcond) \sz_N^2}}{1+ \sz_N^2}\right)^2 \sz_N^2 \\
    & = \frac{\invcond (1 + \sz_N^2)^2 - \sz_N^2\left(1-\sqrt{\invcond - (1 - \invcond) \sz_N^2}\right)^2}{(1 + \sz_N^2)^2} \\
    & = \left(\frac{\sz_N^2 + \sqrt{\invcond - (1 - \invcond) \sz_N^2}}{1 + \sz_N^2}\right)^2.
\end{align*}
We get
\begin{align*}
    \alpha_{N+1} \sz_N^2 + \alpha_{N+1} - 1
    & = \frac{(1 + \sz_N^2)\sqrt{\invcond - \sz_{N+1}^2}}{\sqrt{\invcond - (1 - \invcond) \sz_N^2}} - 1 \\
    & = \frac{\sz_N^2 + \sqrt{\invcond - (1 - \invcond) \sz_N^2}}{\sqrt{\invcond - (1 - \invcond) \sz_N^2}} - 1 \\
    & = \frac{\sz_N^2}{\sqrt{\invcond - (1 - \invcond) \sz_N^2}} \geq 0
\end{align*}
and
\begin{align*}
    & \alpha_{N+1} \sz_N^2 - \sz_{N+1}^2 - \invcond (\alpha_{N+1} - 1) \\
    & = - \alpha_{N+1} (\invcond - \sz_N^2) + \invcond - \sz_{N+1}^2 \\
    & = -\sqrt{\frac{\invcond - \sz_{N+1}^2}{\invcond - (1 - \invcond) \sz_N^2}} (\invcond - \sz_N^2) + (\invcond - \sz_{N+1}^2) \\
    & = \sqrt{\frac{\invcond - \sz_{N+1}^2}{\invcond - (1 - \invcond) \sz_N^2}} \left(-(\invcond - \sz_N^2) + \sqrt{\invcond - \sz_{N+1}^2}\left(\sqrt{\invcond - (1 - \invcond) \sz_N^2} - \sz_N^2\right) + \sz_N^2 \sqrt{\invcond - \sz_{N+1}^2} \right) \\
    & = \sqrt{\frac{\invcond - \sz_{N+1}^2}{\invcond - (1 - \invcond) \sz_N^2}} \left(-(\invcond - \sz_N^2) + \frac{\left(\sz_N^2 + \sqrt{\invcond - (1 - \invcond) \sz_N^2}\right)\left(\sqrt{\invcond - (1 - \invcond) \sz_N^2} - \sz_N^2\right)}{1 + \sz_N^2}\right) \\&\qquad + \frac{\sz_N^2 (\invcond - \sz_{N+1}^2)}{\sqrt{\invcond - (1 - \invcond) \sz_N^2}} \\
    & = \sqrt{\frac{\invcond - \sz_{N+1}^2}{\invcond - (1 - \invcond) \sz_N^2}} \left(-(\invcond - \sz_N^2) + \frac{(\invcond - \sz_N^2)(1 + \sz_N^2)}{1 + \sz_N^2}\right) + \frac{\sz_N^2 (\invcond - \sz_{N+1}^2)}{\sqrt{\invcond - (1 - \invcond) \sz_N^2}} \\
    & = \frac{\sz_N^2 (\invcond - \sz_{N+1}^2)}{\sqrt{\invcond - (1 - \invcond) \sz_N^2}} \geq 0.
\end{align*}

In view of the identities above and the definitions of $\hat \gamma_N^2$ and $\hat \delta_N^2$ we immediately get
\begin{align}
    & \hat \gamma_N^2 = \frac{\invcond \sz_N^2}{\sqrt{\invcond - (1 - \invcond) \sz_N^2}}, \label{E:gammaDefB} \\
    & \hat \delta_N^2 = \frac{\sz_N^2 (\invcond - \sz_{N+1}^2)}{\invcond \sqrt{\invcond - (1 - \invcond) \sz_N^2}}. \label{E:deltaDefB}
\end{align}

\subsection{$0\leq \invcond \hat \delta_i \leq \hat \gamma_i \leq \hat \delta_i, \quad i=0,\dots,N+1$}\label{S:gammaDeltaIneq}
For $i=0,\dots,N-1$ the inequalities are immediate from the assumptions on $\gamma_i$, $\delta_i$.
For $i=N$, $0\leq \invcond \hat \delta_N \leq \hat \gamma_N$ is immediate and $\hat \gamma_N \leq \hat \delta_N$ follows from~\eqref{E:gammaDefB} and~\eqref{E:deltaDefB} by
\[
    \frac{\invcond - \sz_{N+1}^2}{\invcond} \geq \frac{\invcond - (1-\invcond)^2 \invcond}{\invcond} = 2\invcond-\invcond^2 \geq \invcond,
\]
where the leftmost inequality follows from the upper bound on $\sz_{i+1}$~\eqref{E:szUpperLowerBound}.
For $i=N+1$ the bound follows directly from the definition.

\subsection{$\hat \gamma_{i+1} \hat \delta_{i+1} = - (\hat \gamma_i - \hat \delta_i)(\hat \gamma_i - \invcond \hat \delta_i), \quad i=0,\dots,N$}
For $i=0,\dots,N-2$ the equality is immediate from the assumptions on $\gamma_i$, $\delta_i$ and the nonnegativity of $\alpha_{N+1}$.

When $i=N-1$, we have from~\eqref{E:gammaDefB} and~\eqref{E:deltaDefB}
\begin{align*}
    \hat \gamma_N \hat \delta_N
    & = \sqrt{\frac{\invcond \sz_N^2}{\sqrt{\invcond - (1 - \invcond) \sz_N^2}} \frac{\sz_N^2 (\invcond - \sz_{N+1}^2)}{\invcond \sqrt{\invcond - (1 - \invcond) \sz_N^2}}} \\
    & = \sqrt{\frac{\invcond - \sz_{N+1}^2}{\invcond - (1 - \invcond) \sz_N^2}} \sz_N^2 \\
    & = \alpha_{N+1} \sz_N^2 
    = - \alpha_{N+1} (\gamma_{N-1} - \delta_{N-1})(\gamma_{N-1} - \hat \delta_{N-1}) \\
    & = - (\hat \gamma_{N-1} - \hat \delta_{N-1})(\hat \gamma_{N-1} - \invcond \hat \delta_{N-1}).
\end{align*}
When $i=N$ we have
\begin{align*}
    & - (\hat \gamma_N - \hat \delta_N)(\hat \gamma_N - \invcond \hat \delta_N) \\
    & = - \hat \gamma_N^2 - \invcond \hat \delta_N^2 + (1+\invcond)  \hat \gamma_N \hat \delta_N \\
    & = - \invcond  \left(\alpha_{N+1} \sz_N^2 + \alpha_{N+1} -1\right) - \left(\alpha_{N+1} \sz_N^2 - \sz_{N+1}^2 - \invcond(\alpha_{N+1} - 1)\right) + (1+\invcond) \alpha_{N+1} \sz_N^2 \\
    & = -(1 + \invcond) \alpha_{N+1} \sz_N^2 + \sz_{N+1}^2 + (1+\invcond) \alpha_{N+1} \sz_N^2 \\
    & =  \sz_{N+1}^2 = \hat \gamma_{N+1} \hat \delta_{N+1}.
\end{align*}

\subsection{$\sum_{i=0}^{N+1} \hat \delta_i^2 = 1$}\label{S:deltaSumEq}
We have
\begin{align*}
    \sum_{i=0}^{N+1} \hat \delta_i^2 
    & = \alpha_{N+1} \sum_{i=0}^{N-1} \delta_i^2 + \invcond^{-1}\left(\alpha_{N+1} \sz_N^2 - \sz_{N+1}^2 - \invcond \alpha_{N+1} + \invcond\right) + \frac{\sz_{N+1}^2}{\invcond} \\
    & = \alpha_{N+1} \sum_{i=0}^{N} \delta_i^2 + \left(1 - \alpha_{N+1}-\frac{\sz_{N+1}^2}{\invcond}\right) + \frac{\sz_{N+1}^2}{\invcond} \\
    & =  1.
\end{align*}

\bibliographystyle{abbrv}
\bibliography{bib}{}   % name your BibTeX data base

\end{document}